\documentclass[11pt]{amsart}
\usepackage{amssymb, latexsym, mathrsfs, color}
\usepackage[colorlinks=true,pdfstartview=FitV,linkcolor=blue,citecolor=blue,urlcolor=blue]{hyperref}

    \advance \topmargin by -\headheight
    \advance \topmargin by -\headsep

    \evensidemargin \oddsidemargin
\newtheorem {theorem}    {Theorem}[section]

\newtheorem {lemma}      [theorem]    {Lemma}

\newtheorem {proposition}[theorem]    {Proposition}

\theoremstyle{definition}
\newtheorem{definition}[theorem]{Definition}
\newtheorem{remark}[theorem]{Remark}
\newtheorem{ex}[theorem]{Example}

\def\Z{\mathbb{Z}}

\def\p{\mathfrak p}

\numberwithin{equation}{section}

\newenvironment{red}{\relax\color{red}}{\relax}
\newenvironment{blue}{\relax\color{blue}}{\hspace*{.5ex}\relax}
\newenvironment{jaune}{\relax\color{green}}{\hspace*{.5ex}\relax}

\newcommand{\ber}{\begin{red}}
\newcommand{\er}{\end{red}}
\newcommand{\beb}{\begin{blue}}
\newcommand{\eb}{\end{blue}}
\newcommand{\bjn}{\begin{jaune}}
\newcommand{\ejn}{\end{jaune}}

\begin{document}

\title[Character expansion of correction factors]{Character expansion of Kac--Moody correction factors}

\date{\today}

\author[K.-H. Lee]{Kyu-Hwan Lee$^{\star}$}
\thanks{$^{\star}$This work was partially supported by a grant from the Simons Foundation (\#318706).}
\address{Department of
Mathematics, University of Connecticut, Storrs, CT 06269, U.S.A.}
\email{khlee@math.uconn.edu}

\author[D. Liu]{Dongwen Liu}
\address{School of Mathematical Sciences, Zhejiang University, Hangzhou 310027, P.R. China}\email{maliu@zju.edu.cn}

\author[T. Oliver]{Thomas Oliver$^{\dagger}$}
\thanks{$^{\dagger}$This article arises from research funded by the John Fell Oxford University Press Research Fund.}
\address{Mathematical Institute, University of Oxford, Andrew Wiles Building, Oxford, OX2 6GG, U.K.}
\email{Thomas.Oliver@maths.ox.ac.uk}


\begin{abstract} 
A correction factor naturally arises in the theory of $p$-adic Kac--Moody groups. In this paper, we expand  the correction factor into a sum of irreducible characters of the underlying Kac--Moody algebra. We derive a formula for the coefficients which lie in the ring of power series with integral coefficients. In the case that the Weyl group is a universal Coxeter group, we show that the coefficients are actually polynomials.
\end{abstract}

\maketitle

\section{Introduction}

Let $W$ be a Coxeter group, and consider its Poincar\'e series
\[ \chi(q):= \sum_{w \in W} q^{\ell(w)} ,\]
where $q$ is an indeterminate and $\ell(w)$ is the length of $w$. R. Steinberg showed in \cite{St} that the series $\chi(q)$ represents a rational function in $q$. When $W$ is the Weyl group of an irreducible, reduced, finite root system $\Phi$, I.G. Macdonald \cite{M72} found the following identity:
\begin{equation} \label{chi-s} \sum_{w \in W} \prod_{\alpha \in \Phi^+}  \left(\frac { 1 - q e^{- w \alpha}}{1 - e^{-w\alpha}}\right) =\chi(q) , \end{equation}
where $\Phi^+$ is the set of positive roots and $e^{\beta}$ is a formal exponential associated to $\beta$ in the root lattice $Q$.  Macdonald's identity reflects the geometry of the flag manifold.

A generalization of the left-hand side of \eqref{chi-s} to a Kac--Moody root system $\Phi$ would be
\[ \mathcal M (q) := \sum_{w \in W} \prod_{\alpha \in \Phi^+}  \left(\frac { 1 - q e^{- w \alpha}}{1 - e^{-w\alpha}}\right)^{m(\alpha)}, \]
where $m(\alpha)$ is the multiplicity of $\alpha$. The identity \eqref{chi-s} is no longer true for $\mathcal{M}(q)$, and so it is interesting to compute the {\em  correction factor} \footnote{A slight modification of this quotient, denoted by $\mathfrak m$, is what Macdonald called the {\em constant term} in the affine case and is also called the ``correction factor" in the literature (see \eqref{eq.MacMm} for a precise definition).} $\mathcal M(q)/ \chi(q)$. Macdonald \cite{M03} computed this quotient for the affine Kac--Moody case. The computation turns out to be equivalent to the Macdonald constant term conjecture \cite{M82}, which was proven by I. Cherednik in \cite{Ch95}. 

 The correction factor appears in the study of $p$-adic affine Kac--Moody groups, namely in the formal computation of Fourier coefficients of Eisenstein series and in the study of corresponding Hecke algebras. For example, it was shown by Braverman--Finkelberg--Kazhdan that this correction factor appears in the Gindikin--Karplevich formula for affine  Kac--Moody groups \cite{BFK} (see also \cite{BGKP,  BKP, GR, BPGR}). The correction factor in the general case was studied by Muthiah--Puskas--Whitehead \cite{MPW}. They encoded the data of the correction factor into a  collection of polynomials indexed by positive imaginary roots and derived formulas for these polynomials.

In this paper, we study the correction factor $\mathcal M(q)/ \chi(q)$ for arbitrary Kac--Moody root systems, which we write as a sum of characters  $\text{\rm ch}\left(L(\lambda)\right)$ of integrable irreducible representations $L(\lambda)$  of the Kac--Moody algebra $\mathfrak g$ with root system $\Phi$. As the first main result of this paper, we prove that the sum is supported on $\lambda \in P^+ \cap Q^-_{\rm im} $, where $P^+$ is the set of dominant integral weights and $Q_{\rm im}^-$ is the cone generated by negative imaginary roots. More precisely, we obtain

\begin{theorem}\label{lem.suppMPQ-1}
Given a Kac--Moody algebra $\mathfrak{g}$, let $P^+$ denote its set of dominant integral weights and  $Q^-_{\rm im}$  its negative  imaginary root cone. Then there are 
$d_\lambda \in \mathbb Z[[q]]$,  $\lambda\in P^+\cap Q^-_{\rm im}$, such that
\begin{equation}\label{eq.clambda-1}
\mathcal{M}(q) / \chi(q) = \sum_{\lambda \in P^+\cap Q^-_{\rm im}} d_\lambda \, \text{\rm ch}\left(L(\lambda)\right) . 
\end{equation}
\end{theorem}

Actually, we prove this result for any $W$-invariant functions with support in the negative root cone $Q^-$ (see Theorem \ref{lem.suppMPQ}). We recover \eqref{chi-s} as an immediate consequence, since $P^+ \cap Q^-_{\rm im}= \{ 0 \}$
for finite root systems. This result also explains why the known formulas in the affine case only involve imaginary roots.

The coefficients $d_\lambda$ are related to the function $H(\mu; q)$, $\mu \in Q$, which was introduced by Kim and Lee \cite{KL11, Kim--Lee} in a study of $p$-adic integrals using canonical/crystal bases from the context of Weyl group multiple Dirichlet series (\cite{Bu} for a survey). See Definition \ref{def-kl} for the definition of $H(\mu;q)$. We prove the following formula (Theorem \ref{thm.formulaford}): 
\begin{equation}\label{eq.formulad-1}
\chi(q) \, d_{\lambda} = \sum_{w \in W} (-1)^{\ell(w)}H(-w \circ \lambda; q), 
\end{equation}
where $w\circ \lambda:= w(\lambda +\rho) -\rho$ with a Weyl vector $\rho$.

Using  \eqref{eq.formulad-1},  one can compute $d_\lambda$ explicitly. In particular, in the rank $2$ hyperbolic case, we observe that they are actually polynomials in $q$. Generalizing this observation, we prove that $d_\lambda$ are always polynomials when $W$ is a universal Coxeter group, or equivalently, when $a_{ij} a_{ji} \ge 4$ for all $i,j \in I$ with the generalized Cartan matrix $A=(a_{ij})_{i,j \in I}$  of $\mathfrak g$. Formally, we obtain

\begin{theorem}\label{thm.polynomial}
Assume that the Weyl group $W$ of $\mathfrak g$ is  a universal Coxeter group. Then we have $d_{\lambda}\in\mathbb{Z}[q]$ for all $\lambda\in P^+\cap Q_{\rm im}^-$.
\end{theorem}

It would be very interesting to see if $d_\lambda$ are polynomials for arbitrary Kac--Moody root systems. We expect that these coefficients carry important combinatorial, representation-theoretic information, which is yet to be revealed. We hope that we can investigate these issues in the near future.  

The main text proceeds as follows. In Section~\ref{section.Background} we review standard background material and construct a large ring containing $\mathcal{M}(q)$ equipped with a $W$-action. We conclude with the statement that $W$-invariant elements admit a character expansion, which applies in particular to $\mathcal{M}(q)$. In Section~\ref{section.formula} we compute the character coefficients in terms of the function $H$. Though the formula deduced involves an infinite sum, it exhibits a large amount of cancellation and in Section~\ref{section.poly} we show that it is in fact a polynomial when $W$ is a universal Coxeter group. In the Appendix, we give compute the coefficients for certain small imaginary roots of a rank 2 hyperbolic Kac--Moody algebra. 

\subsection*{Acknowledgments} We thank Dinakar Muthiah, Anna Pusk\'as and Ian Whitehead for helpful discussions, and are grateful to the anonymous referee for helpful comments.

\section{Existence of character coefficients}\label{section.Background}
We will use the conventions and terminology of \cite{Kac}. Let $I=\{1,\cdots,n\}$ and let $A$ be a generalized Cartan matrix with realisation $(\mathfrak{h},\Pi,\Pi^{\vee})$. In particular, the elements of the set $\Pi=\{\alpha_1,\dots,\alpha_n\} \subset \mathfrak h^*$ (resp. $\Pi^{\vee}=\{\alpha_1^{\vee},\dots,\alpha_n^{\vee}\} \subset \mathfrak h$)  are the simple roots (resp. simple coroots). The root lattice $Q$ (resp. positive root cone $Q^+$) is the $\mathbb{Z}$-span (resp. $\mathbb{Z}_{\geq0}$-span)  of  $\Pi$. We set $Q^-=-Q^+$. A partial order $\geq$ on $\mathfrak{h}^{\ast}$ is defined by $\mu\geq\nu$ if $\mu-\nu\in Q^+$. We say $\alpha\in\mathfrak{h}^{\ast}$ is positive (resp. negative) if $\alpha>0$ (resp. $\alpha<0$). 

Let $\mathfrak{g}$ be the Kac--Moody algebra associated to $A$, which admits the root space decomposition $\mathfrak{g}=\oplus_{\alpha\in Q}\mathfrak{g}_{\alpha}$, with $\mathfrak{g}_0=\mathfrak{h}$. Given $\alpha\in Q$, its multiplicity $m(\alpha)$ is the dimension of the vector space $\mathfrak{g}_{\alpha}$. A non-zero $\alpha\in Q$ is a root if $m(\alpha)\neq0$. We will denote the set of roots by $\Phi$, and the set of positive (resp. negative) roots by $\Phi^+$ (resp. $\Phi^-$).  

Let $W$ denote the Weyl group of $\mathfrak g$, which is the subgroup of $\mathrm{Aut}\left(\mathfrak{h}^{\ast}\right)$ generated by the simple reflections $s_i$, $i\in I$. A root $\alpha\in\Phi$ is called real if there is $w\in W$ such that $w\alpha$ is a simple root. A root that is not real is called imaginary. If $\alpha$ is real, then $m(\alpha)=1$. The set of real (resp. imaginary) roots is denoted by $\Phi_{\mathrm{re}}$ (resp. $\Phi_{\mathrm{im}}$), and the set of positive real (resp. positive imaginary) roots is denoted by $\Phi^+_{\mathrm{re}}$ (resp. $\Phi^+_{\mathrm{im}}$). 

Let $q$ denote a formal variable, and let $\mathbb{Z}[[q]]$ be the ring of power series in $q$ with integer coefficients. Recall that $f(q)\in\mathbb{Z}[[q]]$ is invertible if and only if the constant term $f(0)$ of $f$ is equal to $\pm1$. The inverse of a unit in $\mathbb Z[[q]]$ will be written as a fraction whenever it is convenient. For example, we write
\[\frac 1 {1-q}= 1+ q+q^2 + \cdots . \]
\begin{ex}\label{ex.Poincare}
The Poincar\'{e} series of the Weyl group $W$ is defined as follows:
\begin{equation}\label{eq.poincare}
\chi(q)=\sum_{w\in W}q^{\ell(w)}\in\mathbb{Z}[[q]],
\end{equation}
where the length $\ell(w)$ of $w\in W$ is the minimal $\ell$ such that $w=s_{i_1}\cdots s_{i_{\ell}}$ is a product of simple reflections. As the only word of length 0 is the identity element, the constant term of $\chi(q)$ is $1$. Thus, $\chi(q) \in\mathbb{Z}[[q]]^{\times}$.
\end{ex}
\textbf{Notation.} To each $\lambda\in\mathfrak{h}^{\ast}$, we associate a formal exponential denoted by $e^{\lambda}$, and define $e^{\lambda}e^{\mu}=e^{\lambda+\mu}$ for $\lambda,\mu\in\mathfrak{h}^{\ast}$.  Let $\Z((q))$ denote the ring of Laurent series with integral coefficients, and let $R$ be a subring of $\mathbb{Z}((q))$. We denote by $\mathcal S(R)$ the additive group of formal sums  $\sum_{\lambda \in \mathfrak{h}^{\ast}}a_{\lambda}e^{\lambda}$ with $a_{\lambda}\in R$ for all $\lambda\in\mathfrak{h}^{\ast}$. 

\begin{definition}
The {\em support} of a formal sum $\sum_{\lambda\in\mathfrak{h}^{\ast}}a_{\lambda}e^{\lambda}\in\mathcal{S}(R)$ is the set of $\lambda\in\mathfrak{h}^{\ast}$ such that $a_{\lambda}\neq0$.
\end{definition}
If $f=\sum_{\lambda \in  Q}a_{\lambda}e^{\lambda}$ is a unit of $\mathcal{S}(R)$ and has support in a translate of  $Q^-$, then $f$ has a unique product expansion as in \cite[Proposition~2.2]{MPW}:
\begin{equation}\label{eq.product}
\sum_{ \lambda \in Q}a_{\lambda}e^{\lambda}={ u e^{\lambda_0}}\prod_{\lambda\in Q^-\backslash\{0\}}\prod_{n\in\Z}(1-q^ne^{\lambda})^{m(\lambda,n)},
\end{equation}
for some $u\in R^\times$, $\lambda_0 \in Q$ and $m(\lambda,n)\in\mathbb{Z}$ such that, for every $\lambda$, the set $\{n\in\Z:m(\lambda,n)\neq0\}$ is bounded below.

\begin{definition}[Section 2.3 in~\cite{MPW}]
A product of the form~\eqref{eq.product} is called a {\em good product with coefficients in $R$} if all $\lambda$ appearing in its factors are multiples of roots $\alpha\in\Phi$, and the set of factors corresponding to any real root $\alpha\in\Phi_{\mathrm{re}}$ is finite. We will denote by $\mathcal{G}(R)$ the multiplicative group of good products with coefficients in $R$.
\end{definition}
An element of $\mathcal G(R)$ expands to a formal sum in $\mathcal S(R)$ by definition. The notion of a good product is introduced, in part, to define the action of $W$ as below.

\begin{definition}\label{def.goodproductaction}
We define an action of $W$ on $\mathcal{G}(R)$ by extending the following action on the factors of \eqref{eq.product} multiplicatively: 
\begin{equation}\label{eq.goodproductaction}
w(1-q^ne^{\lambda})=\begin{cases}
1-q^ne^{w(\lambda)},&w(\lambda)<0,\\
(-q^ne^{w(\lambda)})(1-q^{-n}e^{-w(\lambda)}),&w(\lambda)>0,
\end{cases}
\end{equation}
for $w\in W$. Given $f\in\mathcal{G}(R)$, we will sometimes write $f^w=w(f)$. We will denote by $\mathcal{G}^W(R)$ the ring of $W$-invariant elements of $\mathcal{G}(R)$. 
\end{definition}

Define the negative imaginary cone $Q_{\rm im}^-$ to be the cone generated by negative imaginary roots. Then we have $Q_{\rm im}^-=\bigcap_{w\in W}w(Q^-)$.  Thus if $f\in \mathcal{G}^W(R)$ is supported on $Q^-$, then it is in fact supported on $Q_{\rm im}^-$. It was noted in \cite{MPW} that, for $w \in W$ and $f=\sum_{\lambda \in \mathfrak{h}^{\ast}}a_{\lambda}e^{\lambda}\in\mathcal{G}(R)$, we have
\begin{equation}\label{eq.seriesaction}
w(f)= \sum_{\lambda \in \mathfrak{h}^{\ast}}a_{\lambda}e^{w\lambda}. 
\end{equation}

\begin{remark}
The set of $f\in\mathcal{S}(R)$ supported on $Q^-$ is not closed under the action of $W$ defined by~\eqref{eq.seriesaction}, but $\mathcal{G}(R)$ is. 
\end{remark}
The basic good product in this paper is
\begin{equation}\label{eq.DeltaDot}
\Delta:=\prod_{\alpha \in \Phi^+}  \left(\frac { 1 - q e^{- \alpha}}{1 - e^{-\alpha}}\right)^{m(\alpha)}.
\end{equation}
Here $\frac { 1 - q e^{- \alpha}}{1 - e^{-\alpha}}= 1+ \sum_{n\ge 1} (1-q) e^{-n \alpha}$, and it is clear that $\Delta \in \mathcal {G} \left(\Z[q]\right)$.

Since $m(\alpha) =1$ for $\alpha \in \Phi^+_{\mathrm{re}}$, we set
\begin{equation}
\Delta_{\mathrm{re}}:=\prod_{\alpha\in\Phi^+_{\mathrm{re}}}\left(\frac { 1 - q e^{- \alpha}}{1 - e^{-\alpha}}\right),\qquad \Delta_{\mathrm{im}}:=\prod_{\alpha\in\Phi^+_{\mathrm{im}}}\left(\frac { 1 - q e^{- \alpha}}{1 - e^{-\alpha}}\right)^{m(\alpha)}
\end{equation} so that we have \[ \Delta=\Delta_{\mathrm{re}}\Delta_{\mathrm{im}}.\] 

Finally, define
\begin{equation}\label{eq.correction}
\boxed{\mathcal{M}(q):=\sum_{w\in W}\Delta^w.}
\end{equation}
Clearly, $\mathcal M(q)$ is $W$-invariant since it is the sum of $W$-action on $\Delta$.

\begin{lemma}\label{ex.Delta}
The formal sum $\mathcal M (q)$ is a $W$-invariant good product with coefficients in $\Z[[q]]$, i.e.
$\mathcal{M}(q)\in\mathcal{G}^W\left(\Z[[q]]\right)$. Moreover, $\mathcal{M}(q)$ is supported on $Q_{\rm im}^-$  and has the constant term equal to $\chi(q)$. 
\end{lemma}

\begin{proof}
Since the set $\Phi^+_{\mathrm{im}}$ is $W$-invariant, and $m(w \alpha) =m(\alpha)$ for $w \in W$ and $\alpha \in \Phi$, we have $\Delta_{\mathrm{im}}^w=\Delta_{\mathrm{im}}$. It follows that $\Delta^w=\left(\Delta_{\text{re}}\Delta_{\text{im}}\right)^w=\Delta^w_{\text{re}}\Delta_{\text{im}}$. By Definition \ref{def.goodproductaction}, we have 
\[
w\left(\frac{1-qe^{-\alpha}}{1-e^{-\alpha}}\right)
=\begin{cases}\displaystyle{
\frac {1-qe^{-w(\alpha)}} {1-e^{-w(\alpha)}} }& \text{if } w(\alpha)>0,\\  \\
\displaystyle {\frac{ qe^{-w(\alpha)}(1-q^{-1}e^{w(\alpha)})} {e^{-w(\alpha)}(1-e^{w(\alpha)} ) }=\frac { q(1-q^{-1}e^{w(\alpha)})}{1-e^{w(\alpha)}} } & \text{if } w(\alpha)<0, 
\end{cases}
\] for $w \in W$ and $\alpha \in \Phi^+_{\mathrm{re}}$. One can immediately see that 
 the sum $\mathcal{M}(q)$ is supported on $Q^-$.
Since \begin{equation} \label{eq-easy} 
\dfrac { 1 - q e^{- \alpha}}{1 - e^{-\alpha}}= 1+ \sum_{n\ge 1} (1-q) e^{-n \alpha} \ \text{ and } \  \dfrac { q(1 - q^{-1} e^{- \alpha})}{1 - e^{-\alpha}}=\dfrac { q -  e^{- \alpha}}{1 - e^{-\alpha}}= q- \sum_{n\ge 1} (1-q) e^{-n \alpha},
\end{equation} 
we see that $\Delta^w$ is a good product with coefficients in $\Z[q]$, i.e. $\Delta^w \in \mathcal G(\Z[q])$. 

Now we check that the coefficient of $e^{-\beta}$ in $\mathcal M(q) =\sum_{w \in W} \Delta^w$ is an element of $\Z[[q]]$ for $\beta \in Q^+$. For $w \in W$, define \[ \Phi(w) := \{ \alpha \in \Phi^+_{\mathrm{re}} \ |\  w (\alpha) <0 \} = \Phi^+ \cap w^{-1} \Phi^-.\] It is well-known that $|\Phi(w)| = \ell(w)$. Thus we have
\begin{align}  
\Delta^w_{\mathrm{re}}&=\prod_{\alpha\in\Phi(w^{-1})}\left(\frac { q -  e^{- \alpha}}{1 - e^{-\alpha}}\right)\prod_{\alpha\in\Phi^+_{\mathrm{re}}\setminus \Phi(w^{-1})}\left(\frac { 1 - q e^{- \alpha}}{1 - e^{-\alpha}}\right) \nonumber \\ &=q^{\ell(w)} \ \prod_{\alpha\in\Phi(w^{-1})}\left(\frac { 1 -q^{-1}  e^{- \alpha}}{1 - e^{-\alpha}}\right)\prod_{\alpha\in\Phi^+_{\mathrm{re}}\setminus \Phi(w^{-1})}\left(\frac { 1 - q e^{- \alpha}}{1 - e^{-\alpha}}\right).\label{eq-easy-1} 
\end{align} 
For $\beta \in Q^+$, the coefficient of $e^{-\beta}$ in $\Delta^w = \Delta^w_{\mathrm{re}} \Delta_{\mathrm{im}}$, a priori an element in $\mathbb{Z}[[q]]$, is of the form
\[ q^{\ell(w)} p_{\beta,w} \] 
for some $p_{\beta,w} \in \Z((q))$. Recall the {\em height} of $\beta=\sum^n_{i=1}m_i\alpha_i\in Q^+$, $m_i\geq 0$, is defined to be
\[
{\rm ht}(\beta):=\sum^n_{i=1}m_i.
\]
It is easy to observe from \eqref{eq-easy} and \eqref{eq-easy-1} the crude estimate  that the degrees of $p_{\beta,w}$ in $q^{-1}$ and $q$ are both bounded by  ${\rm ht}(\beta)$.  Thus we have $p_{\beta, w}\in \mathbb{Z}[q, q^{-1}]$. Moreover $q^m$ appears in $q^{\ell(w)} p_{\beta,w}$ only if $\ell(w)\leq m+{\rm ht}(\beta)$. Since there are only finitely many $w\in W$ of a given length,  we see that
 \[
 \mathcal{M}(q)=\sum_{w\in W}\Delta^w=\sum_{\beta\in Q^+}\left(\sum_{w\in W} q^{\ell(w)}p_{\beta, w}\right)e^{-\beta}
 \]
 with the coefficient of $e^{-\beta}$ given by a well-defined sum
 \[
 \sum_{w\in W} q^{\ell(w)}p_{\beta, w}\in \mathbb{Z}[[q]].
 \]
In particular, when $\beta =0$, we have $p_{0,w}=1$ for all $w \in W$ and the constant term of $\mathcal M(q)$ is equal to $\sum_{w \in W} q^{\ell(w)} = \chi(q)$. 

We have already seen that $\mathcal M(q)$ is supported on $Q^-$ at the beginning of the proof. Since $\mathcal M(q)$ is also $W$-invariant, it is supported on $Q^-_{\mathrm{im}}$. (See the paragraph after Definition \ref{def.goodproductaction}.)

Using \cite[Proposition~2.2]{MPW}, we may write $\mathcal{M}(q)$ as a product of the form~\eqref{eq.product} with $\lambda_0=0$. Since $\mathcal M(q)$ is supported on $Q^-_{\mathrm{im}}$, no factor corresponding to a real root arises in the product and hence $\mathcal M(q)$ is a good product.
\end{proof}

\begin{remark}
(1) We have the following identity in $\mathcal{G}^W \left(\Z[[q]]\right)$: 
\begin{equation}\label{eq.MacMm}
\mathfrak{m}\mathcal{M}(q)=\Delta_{\text{im}}\chi(q),
\end{equation}
where $\mathfrak{m}$ is  as defined in \cite[equation~(3.5)]{MPW}.
Each of $\frak{m}^{-1}$, $\Delta_{\rm im}$ and $\mathcal{M}(q)$ expands to a formal sum supported on $Q_{\rm im}^-$. 

(2) In the paper \cite{BPGR2}, it was pointed out that $\mathcal M(q)$ is not an element of $\mathcal G^W(\Z[q,q^{-1}])$ but an element of $\mathcal G^W(\Z((q)))$. As a refinement, Lemma \ref{ex.Delta} shows that $\mathcal M(q) \in \mathcal G^W(\Z[[q]])$.

\end{remark}

\medskip

Now we move on to study a character expansion of an element in $\mathcal G^W(\Z[[q]])$.

\begin{definition}\label{def.circleaction}
Fix a Weyl vector $\rho\in\mathfrak{h}^{\ast}$, i.e. a vector satisfying $\rho(\alpha_i^\vee)=1$, for all $i\in I$.  The circle action\footnote{This action is slightly different to the action with the same notation in \cite{Kim--Lee}.} of $W$ on $\mathfrak h^\ast$ is defined by
\begin{equation}\label{eq.circleaction}
w\circ\lambda=w(\lambda+\rho)-\rho.
\end{equation}  
\end{definition}
\begin{ex}
We have 
\begin{equation}\label{eq.wcirc0}
 w \circ 0 = w \rho - \rho,  
\end{equation}
which can be written as a sum of  negative roots. Indeed, one has
\begin{equation}\label{eq.rhominuswrho}
\rho-w\rho=\sum_{\alpha \in \Phi(w^{-1})}\alpha,
\end{equation}
where,  for $w\in W$, 
\begin{equation}\label{eq.Phiw}
\Phi(w):=\Phi^+\cap w^{-1}\Phi^-.
\end{equation}
\end{ex}
Denote by $P$ the weight lattice of $\mathfrak g$, and by $P^+\subset P$ the subset of dominant integral weights. For $\lambda \in P$, define
\[ 
\pi^\lambda := \frac {\sum_{w \in W} (-1)^{\ell(w)} e^{w(\lambda+\rho)}}{\sum_{w \in W} (-1)^{\ell(w)} e^{w\rho}} . 
\]
Recall the denominator identity
\begin{equation}\label{eq.denominatoridentity}
 \sum_{w \in W} (-1)^{\ell(w)} e^{w\rho -\rho} = \prod_{\alpha \in \Phi^+} (1-e^{-\alpha})^{m(\alpha)}.
 \end{equation}
For $\lambda \in P$, define
\begin{equation} \label{eqn-def-xi} \xi^\lambda : = \sum_{w \in W} (-1)^{\ell(w)} e^{w \circ \lambda} .\end{equation}

\begin{lemma} \label{lem-w-circ} \hfill

(1) For $w \in W$, we have 
\[ w \left(\prod_{\alpha \in \Phi^+} (1-e^{-\alpha})^{m(\alpha)}\right)  = (-1)^{\ell(w)} e^{\rho-w\rho} \prod_{\alpha \in \Phi^+} (1-e^{-\alpha})^{m(\alpha)} .\]

(2) For $\lambda \in P$ and $w \in W$, we have 
\begin{align*}
\xi^\lambda &= (-1)^{\ell(w)} \xi^{w \circ \lambda},\\ 
\pi^\lambda &= (-1)^{\ell(w)}\pi^{w \circ \lambda}.
\end{align*}

\end{lemma}

\begin{proof}
(1) From the denominator identity \eqref{eq.denominatoridentity}, we have 
\begin{align*} 
w \left(\prod_{\alpha \in \Phi^+} (1-e^{-\alpha})^{m(\alpha)}\right) &= \sum_{w_1 \in W} (-1)^{\ell(w_1)} e^{ww_1\rho -w\rho}  \\ 
&= \sum_{w_1 \in W} (-1)^{\ell(w)+\ell(ww_1)} e^{ww_1\rho -\rho} e^{\rho -w\rho}  \\ 
& = (-1)^{\ell(w)} e^{\rho-w\rho} \sum_{w_1 \in W} (-1)^{\ell(ww_1)} e^{ww_1\rho -\rho}   \\ 
&=(-1)^{\ell(w)} e^{\rho-w\rho} \prod_{\alpha \in \Phi^+} (1-e^{-\alpha})^{m(\alpha)}.  
\end{align*}

(2) Let $w \circ \lambda = \mu$. Then  $w(\lambda + \rho) = \mu +\rho$. Now we have
\begin{align*}
\sum_{w_1 \in W} (-1)^{\ell(w_1)} e^{w_1 (\lambda+\rho)} &= \sum_{w_1 \in W} (-1)^{\ell(w)+\ell(w_1w^{-1})} e^{w_1w^{-1} w(\lambda+\rho)} \\
&= (-1)^{\ell(w)} \sum_{w_1 \in W} (-1)^{\ell(w_1w^{-1})} e^{w_1w^{-1} (\mu+\rho)}.
\end{align*}
Multiplying both sides by $e^{-\rho}$, we get  $\xi^\lambda = (-1)^{\ell(w)} \xi^{\mu} = (-1)^{\ell(w)} \xi^{w \circ \lambda}$. Dividing both sides by $\sum_{w_1 \in W} (-1)^{\ell(w_1)} e^{w_1\rho-\rho}$, we obtain $\pi^\lambda = (-1)^{\ell(w)} \pi^{w \circ \lambda}$.
\end{proof}
Consider the following subset of $Q^-$:
\begin{equation}\label{eq.Qprime}
Q':=\bigcap_{w\in W}w\circ Q^-.
\end{equation}
The Weyl group $W$ acts on $Q'$ by the circle action, and so $Q_{\rm im}^-\subset Q'$. 

\begin{lemma} \label{lem-refl} 
{ Assume that $\lambda\in Q'$.  Then the following hold.} 
\begin{enumerate}
\item  There exists  a unique $\mu \in Q^-$ and $v \in W$ such that $\mu +\rho \in P^+$ and $v \circ \lambda = \mu$.
\item  The stabilizer subgroup
\[ W_\lambda^\circ :=\{ w \in W : w\circ \lambda = \lambda \} \]
is generated by reflections in $W$. 
\item If $\lambda \in P^+ \cap Q^-$, then  $W_\lambda^\circ = \{ 1 \}$. 
\end{enumerate}
\end{lemma}

\begin{proof}
Write $\lambda = \sum_i m_i \alpha_i$ with $m_i \le 0$ for all $i$. If $\lambda +\rho \in P^+$, there is nothing to prove. If not, there exists $j$ such that $\lambda(\alpha^{\vee}_j) \le -2$. We have \[ s_j \circ \lambda = s_j(\lambda +\rho)-\rho = \lambda - (\lambda(\alpha^{\vee}_j) +1) \alpha_j  \in Q^-.  \]
Since $\lambda(\alpha^{\vee}_j)+1 <0$, we have $m_j < m_j -(\lambda(\alpha^{\vee}_j)+1) \le 0$. If $(s_j \circ \lambda) +\rho$ is in $P^+$, we are done. Otherwise, repeat the process with replacing $\lambda$ with $s_j \circ \lambda$. Since the coefficients are increasing and bounded above by $0$, this process must end.

Assume that $\lambda +\rho \in P^+$. Suppose that $w \circ \lambda = \mu$ and  $\mu +\rho \in P^+$  for $w=s_{i_1} s_{i_2} \cdots s_{i_\ell} \neq 1$, a reduced expression.  Then we have $w(\lambda+\rho)=\mu+\rho$. Since $(\lambda+\rho)(\alpha^{\vee}_{i_\ell}) \ge 0$, we have $(\mu +\rho) (w (\alpha^{\vee}_{i_\ell})) \ge 0$. Since $w=s_{i_1} s_{i_2} \cdots s_{i_\ell}$ is a reduced expression, we get $w(\alpha^{\vee}_{i_\ell}) <0$, and $(\mu +\rho)(w(\alpha^{\vee}_{i_\ell}))\le 0$. Thus $(\mu +\rho)(w(\alpha^{\vee}_{i_\ell}))=0$ and $(\lambda +\rho) (\alpha^{\vee}_{i_\ell})=0$. Hence $s_{i_\ell}(\lambda+\rho) =\lambda +\rho$ and $s_{i_\ell} \circ \lambda =\lambda$. By induction, we obtain $\mu=\lambda$, which completes a proof of (1).  We have also shown that the subgroup $W_\lambda^\circ$ is generated by simple reflections for $\lambda +\rho \in P^+$. 

Assume that $\lambda \in P^+ \cap Q^-$, and suppose that $w \circ \lambda = \lambda$ for $w=s_{i_1} s_{i_2} \cdots s_{i_\ell} \neq 1$, a reduced expression. Then $(\lambda+\rho)(\alpha^{\vee}_{i_\ell}) >0$ and the above argument leads to a contradiction. Thus we must have $w=1$. This proves (3).

Now assume that { $\lambda \in Q'$. } By Lemma \ref{lem-refl}(1), there exists $v \in W$ such that $v \circ \lambda +\rho \in P^+$. Then $W^\circ_{v \circ \lambda}$ is generated by simple reflections $s_i$. Hence $W^\circ_\lambda$ is generated by $v^{-1} s_i v$, which are reflections. This completes a proof of (2).
\end{proof}

\begin{lemma}\label{lem.stabilizer}
Assume that $\lambda\in Q'$.  The series $\xi^\lambda \in\mathcal{G}(\mathbb{Z})$ defined in \eqref{eqn-def-xi}  is non-zero if and only if the stabilizer subgroup $W_\lambda^\circ$ of $\lambda$ under the circle action is trivial.
\end{lemma} 
\begin{proof}
Suppose that $\xi^\lambda =0$. Then the term $e^{\lambda}$ cancels with $(-1)^{\ell(w)} e^{w \circ \lambda}$ for some $w \neq 1$. In particular, $\lambda=w\circ \lambda$, and the stabilizer subgroup $W_\lambda^\circ$ is not trivial. 

Conversely, assume that the stabilizer subgroup $W_\lambda^\circ$ is not trivial. By Lemma \ref{lem-refl} there exists a reflection $s \in W^\circ_\lambda$ such that $s \circ \lambda =\lambda$. It follows from Lemma \ref{lem-w-circ} that $\xi^\lambda = (-1)^{\ell(s)} \xi^{s \circ \lambda} =- \xi^{\lambda}$. Hence $\xi_\lambda=0$.
\end{proof}
Given $\lambda\in P^+$, let $L(\lambda)$ denote the irreducible highest weight module of $\mathfrak{g}$ with highest weight $\lambda$. The module $L(\lambda)$ admits a weight space decomposition $L(\lambda)=\oplus_{\mu\in\mathfrak{h}^{\ast}}L_{\mu}$. The character $\text{ch}(L(\lambda))$ of $L(\lambda)$ is defined by
\begin{equation}\label{eq.character}
\text{ch}(L(\lambda))=\sum_{\mu \in\mathfrak{h}^{\ast}}\left(\dim L_{\mu}\right)e^{\mu}.
\end{equation}
If $\lambda \in P^+$, then by \cite{Kac} we have
\begin{equation}\label{eq.pilambda}
\pi^\lambda = \text{ch}\left(L(\lambda)\right) .
\end{equation}
Theorem~\ref{lem.suppMPQ-1} is a consequence of the following result.
\begin{theorem}\label{lem.suppMPQ}
Given a Kac--Moody algebra $\mathfrak{g}$, let $P^+$ denote its set of dominant integral weights and $Q^-_{\rm im}$  its negative imaginary root cone. If $f\in\mathcal{G}^W\left(\Z[[q]]\right)$ is such that $\mathrm{supp}(f)\subset Q^-$, then there are 
$c_\lambda \in \mathbb Z[[q]]$, $\lambda\in P^+\cap Q^-_{\rm im}$, such that
\begin{equation}\label{eq.fchar}
f =\sum_{\lambda \in P^+\cap Q^-_{\rm im}} c_\lambda \, \text{\rm ch}\left(L(\lambda)\right) . 
\end{equation} 
\end{theorem}
\begin{proof}
Since $f$ is supported on $Q^-$, we may write the following product as a sum supported on $Q^-$:
\begin{equation}\label{eq.1}
\Xi=f\cdot\prod_{\alpha\in\Phi^+}(1-e^{-\alpha})^{m(\alpha)}=\sum_{\beta\in Q^-}c_{\beta}e^{\beta}.
\end{equation}
As $f$ is invariant under $W$, it follows from Lemma \ref{lem-w-circ}(1) that 
\begin{multline} 
w\left(f\cdot\prod_{\alpha\in\Phi^+}(1-e^{-\alpha})^{m(\alpha)}\right)= \sum_{\beta\in Q^-}c_{\beta}e^{w\beta}\\
= (-1)^{\ell (w)} e^{\rho - w\rho} f\cdot\prod_{\alpha\in\Phi^+}(1-e^{-\alpha})^{m(\alpha)}=  \sum_{\gamma \in Q^-} (-1)^{\ell(w)} c_{\gamma}e^{\rho -w \rho+\gamma}.
\end{multline}
Comparing coefficients, we see that for $\beta \in Q^-$,
\[c_{\beta}=(-1)^{\ell(w)}c_{w\circ\beta}. \] 
Moreover, $c_\beta =0$ unless { $\beta\in Q'$, i.e. $\Xi$ is supported on $Q'$. } If $\lambda +\rho \in P^+$ and $\lambda \not \in P^+ \cap Q^-$ for $\lambda \in Q^-$, then there exists $\alpha^{\vee}_i$ such that $(\lambda+\rho) (\alpha^{\vee}_i) =0$ and $s_i\circ \lambda = \lambda$. Thus $\xi^\lambda =0$ by Lemma~\ref{lem.stabilizer}.

By Lemma~\ref{lem-refl}(1) and the above argument, we group the terms of equation~\eqref{eq.1} to get a sum over  $P^+\cap Q^-$, { which is the subset of representatives $\lambda$ of the $\circ$-action of $W$ on $Q'$ such that $\xi^\lambda\neq 0$:}
\begin{equation}\label{eq.2}
\Xi= \sum_{\beta\in Q'}c_{\beta}e^{\beta}=\sum_{\lambda\in P^+ \cap Q^- }  c_{\lambda}\xi^{\lambda}.
\end{equation}
On the other hand, for $\lambda \in P^+$, Weyl's character formula implies
\begin{equation}\label{eq.WCF}
\xi^{\lambda}=\text{ch}\left(L(\lambda)\right) \prod_{\alpha\in\Phi^+}(1-e^{-\alpha})^{m(\alpha)} .
\end{equation}
The result follows from combining \eqref{eq.2} with \eqref{eq.WCF}, noting that $f$ is in fact supported on $Q^-_{\rm im}$.
\end{proof}

\begin{remark}

As mentioned in the introduction, we recover \eqref{chi-s} as an immediate consequence of Theorem \ref{lem.suppMPQ}, since $P^+ \cap Q^-_{\rm im}= \{ 0 \}$
for finite root systems. In the affine case, we have $P^+ \cap Q^-_{\rm im}= \mathbb Z_{\leq 0} \cdot\delta$ with the minimal positive imaginary root $\delta$, and  the theorem shows that the right-hand side of \eqref{eq.fchar} only involves imaginary roots.

\end{remark}

\section{A formula for the character coefficients}\label{section.formula}

In this section, we derive a formula for the coefficients in the expansion of $\mathcal M(q)$ into a sum of characters. We begin with the definition of a function which will play an important role in what follows.

\begin{definition}[\cite{KL11, Kim--Lee}] \label{def-kl}
The function $H:Q^+\rightarrow\mathbb{Z}[q]$ is defined by the generating series in $\mathcal{G}(\mathbb{Z}[q])$:
\begin{equation}\label{eq.Hdef}
\sum_{\mu \in Q^+} H(\mu;q) e^{-\mu} = \prod_{\alpha \in \Phi^+} (1-q e^{-\alpha})^{m(\alpha)} ,
\end{equation}
where $m(\alpha)$ is the multiplicity of $\alpha$.
When we do not need to specify $q$, we will frequently write $H(\mu)=H(\mu;q)$ .
\end{definition}

\begin{remark}
In \cite{KL11, Kim--Lee}, the function $H$ was denoted by $H_\rho$. See (2-13) in \cite{Kim--Lee}. 

\end{remark}

\begin{definition}\label{def.admissible}
Let $\mu \in Q^+$, and $\mathscr P:= \{ (\alpha; i) : \alpha \in \Phi^+, i=1, 2, \dots , m(\alpha) \}$. An {\em admissible partition} of $\mu$ is a finite set $\mathfrak p \subset \mathscr P$ such that $\sum_{(\alpha, i) \in \mathfrak p} \alpha =\mu$. Let $\mathcal P(\mu)$ be the set of admissible partitions of $\mu$. Given $\mathfrak p\in\mathcal P(\mu)$, we will refer to an element $(\alpha,i)\in\mathfrak{p}$ as {\em part} of $\mathfrak{p}$, and denote the number of parts in $\mathfrak p$ by $|\mathfrak p|$. 
\end{definition}
Examples of admissible partitions are given in Appendix \ref{section.rank2}.
\begin{lemma}
We have 
\begin{equation} \label{eqn-H-formula}
H(\mu) = \sum_{\mathfrak{p}\in\mathcal{P}(\mu)} (-q)^{|\mathfrak p|}.
\end{equation} 
\end{lemma}
\begin{proof}
Equation~\eqref{eqn-H-formula} follows from expanding the product in equation~\eqref{eq.Hdef} and computing the coefficient of $e^{-\mu}$.
\end{proof}

We now prove equation \eqref{eq.formulad-1}, which we state below as a theorem for ease of reference.
\begin{theorem}\label{thm.formulaford} 
For $\lambda \in P^+\cap Q^-_{\rm im}$, define $d_{\lambda}$ by equation~\eqref{eq.clambda-1}. Then we have 
\begin{equation}\label{eq.formulad}
\chi(q) \, d_{\lambda} = \sum_{w \in W} (-1)^{\ell(w)}H(-w \circ \lambda). 
\end{equation}
\end{theorem}

\begin{proof}
By definition, we have
\begin{align*}
\mathcal{M}(q)&=\sum_{w \in W} \Delta^w = \sum_{w \in W} \prod_{\alpha \in \Phi^+} \frac{(1-qe^{-w \alpha})^{m(\alpha)}}{(1-e^{-w \alpha})^{m(\alpha)} } \\
&=  \sum_{w \in W}  \frac{\sum_{\mu \in Q^+} H(\mu) e^{-w\mu}}{\prod_{\alpha \in \Phi^+} (1-e^{-w \alpha})^{m(\alpha)} } \\
&= \sum_{ \mu \in Q^-} H(-\mu) \sum_{w \in W} \frac{e^{w \mu}}{\prod_{\alpha \in \Phi^+} (1-e^{-w\alpha})^{m(\alpha)} }.
\end{align*}
Using Lemma \ref{lem-w-circ} (1), we deduce that
\[
\mathcal{M}(q)=\frac{1}{\prod_{\alpha\in\Phi^+}(1-e^{-\alpha})^{m(\alpha)}}\sum_{\mu \in Q^-}\sum_{w\in W}(-1)^{\ell(w)}H(-\mu)e^{w\circ\mu}.
\]
As in the proof of Theorem \ref{lem.suppMPQ}, put
\[
\Xi:=\mathcal{M}(q)\prod_{\alpha\in\Phi^+}(1-e^{-\alpha})^{m(\alpha)}=\sum_{\mu \in Q^-}\sum_{w\in W}(-1)^{\ell (w)}H(-\mu)e^{w\circ\mu}.
\]
Since $\Xi$ is supported on $Q'$, we may rewrite  the above double sum as
\[
\Xi=\sum_{\beta\in Q'}\sum_{w\in W}(-1)^{\ell(w)}H(-w\circ\beta) e^{\beta}.
\]
The theorem then follows from (\ref{eq.2}).
\end{proof}

\begin{ex}\label{ex.c0}
Given $w\in W$,  write $w=s_{i_1} \cdots s_{i_\ell}$ as a reduced expression.  If $\Phi(w^{-1})$ is as defined in equation~\eqref{eq.Phiw}, then
\begin{equation}
\Phi(w^{-1}) =\left \{ \alpha_{i_1}, s_{i_1}(\alpha_{i_2}), \dots , s_{i_1} \cdots s_{i_{\ell-1}}(\alpha_{i_\ell})  \right\} , 
\end{equation}
and
\begin{equation}
 w \circ 0=\rho - w \rho = \sum_{\alpha \in \Phi(w^{-1})} \alpha = \alpha_{i_1}+ s_{i_1}(\alpha_{i_2})+ \cdots + s_{i_1} \cdots s_{i_{\ell-1}}(\alpha_{i_\ell}). \label{eqn-w0-rho}
\end{equation}
Suppose that
\[ 
\rho - w \rho = \beta_1 + \beta_2 +\cdots + \beta_k 
\]
for some positive roots $\beta_1, \dots , \beta_k \in \Phi^+$. Note that we have 
\[ s_{i_1} (\rho -w\rho)  \not \in Q^+ .\]
Since $s_{i_1}$ keeps $\Phi^+$ except $\alpha_{i_1}$, one of the $\beta_i$'s must be equal to $\alpha_{i_1}$. Then $s_{i_1}(\rho - w \rho -\alpha_{i_1})$ is equal to $\rho - w' \rho$ where $w' =s_{i_2}\cdots s_{i_\ell}$. Arguing by induction on $\ell(w)$, we deduce that~\eqref{eqn-w0-rho} is the unique decomposition of $w\circ0$ into a sum of positive roots. Now it follows from \eqref{eqn-H-formula} that
\begin{equation}
H(-w \circ 0) = H(\rho -w\rho) = (-q)^{\ell(w)},
\end{equation}
and so the formula \eqref{eq.formulad} yields \[ d_0=1 . \]
\end{ex}

\begin{lemma}\label{lem.r2coeffsum0}
For all nonzero $\lambda\in P^+ \cap Q^-$ and $w\in W$, the coefficients of $H(-w\circ\lambda)$ sum to zero.
\end{lemma}   

\begin{proof}
From \cite[Lemma~3.18]{Kim--Lee}, we have
\begin{equation} \label{eqn-kl}
H(\mu;1)=
\begin{cases} 
(-1)^{\ell(w)}, & \text{if } \rho-w\rho =\mu \text{ for some } w \in W, \\ 
0, & \text{otherwise}.
\end{cases}
\end{equation}
Therefore, it suffices to show that 
\begin{equation}\label{eq.suffices} 
- w \circ \lambda = - (w (\lambda+\rho)-\rho) =\rho - w(\lambda +\rho) \neq \rho - v \rho 
\end{equation} 
for any $v \in W$. Equation~\eqref{eq.suffices} is equivalent to $\lambda+\rho\neq w^{-1}v\rho$, and so it is enough to show, for any $v\in W$,
\[ 
\lambda \neq v \rho -\rho .
\]
If $v=1$ there is nothing to prove. Consider an arbitrary $v \neq 1$, and write $v^{-1}$ as a reduced word $s_{i_1} \cdots s_{i_k}$. Then we have
\[ 
\rho (v^{-1} \alpha^{\vee}_{i_k}) <0 , 
\]
and 
\[ 
(v\rho-\rho) (\alpha^{\vee}_{i_k}) =v\rho(\alpha^{\vee}_{i_k}) -\rho(\alpha^{\vee}_{i_k}) = \rho (v^{-1}\alpha^{\vee}_{i_k}) -1 \le -2 . 
\]
Thus $v\rho -\rho \notin P^+$. Since $\lambda \in P^+$, we have $\lambda \neq v \rho -\rho$. 
\end{proof}

\begin{definition}
Let $\lambda\in Q^-$ and 
\begin{equation}\label{eq.frakp}
\mathfrak p=\{ (\beta_1; m_1), (\beta_2; m_2), \dots , (\beta_t; m_t)  \}\in\mathcal{P}(-\lambda).
\end{equation}
Given $w \in W$, we define
\begin{equation}\label{eq.mpw}
m(\mathfrak p,w) = t-2 \times  \#\{(\beta_i; j) \in \mathfrak p \, : \,  w\beta_i<0\}. 
\end{equation}
\end{definition}

With $\frak p$ as in equation~\eqref{eq.frakp}, we define 
\[ 
\phi_i(\mathfrak p) := \begin{cases} \{ (s_i\beta_1; m_1),  \dots , (s_i\beta_t; m_t) , (\alpha_i; 1) \}, & \text{ if $\beta_j \neq \alpha_i$ for any $j$}, \\  
\{ (s_i\beta_1; m_1), \dots, (s_i\beta_{j-1}; m_{j-1}), (s_i\beta_{j+1}; m_{j+1}), \dots , (s_i\beta_{t}; m_{t})  \}, & \text{ if $\beta_j = \alpha_i$ for some $j$}. \end{cases} 
\] 
Since \begin{equation*} \label{eqn-siw} - s_i \circ \lambda = - s_i(\lambda +\rho) +\rho = - s_i(\lambda) +\alpha_i , \end{equation*}
we see that $\phi_i(\mathfrak p)\in\mathcal{P}(-s_i \circ \lambda)$. In other words, $\phi_i$ defines a map $\mathcal P(-\lambda) \rightarrow \mathcal P(-s_i \circ \lambda)$.  Replacing $\lambda$ with $s_i\circ \lambda$, we obtain similarly a map from $\mathcal{P}(-s_i\circ \lambda)$ to $\mathcal{P}(- \lambda)$. One can check that these maps are inverses to each other, and so the map $\phi_i$ is a bijection from $\mathcal{P}(-\lambda)$ to $\mathcal{P}(-s_i \circ \lambda)$.

\begin{lemma} \label{lem-bij-mphi}

If $\ell(ws_i) =\ell(w)+1$, then
\[ 
m(\phi_i(\mathfrak p), w) = m(\mathfrak p, w s_i) +1 . 
\]

\end{lemma}

\begin{proof}
 Consider $\phi_i(\mathfrak p)=\{ (\beta_1; m_1),  \dots , (\beta_t; m_t)  \}\in\mathcal{P}(- \lambda)$. First assume that   \[ \mathfrak p =\{ (s_i\beta_1; m_1),  \dots , (s_i\beta_t; m_t) , (\alpha_i; 1) \}. \]
By applying $ws_i$ to the first components, we get $w \beta_1, \dots, w\beta_t, - w \alpha_i$. Since $-w \alpha_i <0$ from the condition $\ell(ws_i) =\ell(w)+1$, we obtain
$m(\mathfrak p , ws_i) =m(\phi_i(\mathfrak p), w) +1-2 = m(\phi_i(\mathfrak p), w)- 1$.  Next assume that  
\[ \mathfrak p = \{ (s_i\beta_1; m_1), \dots, (s_i\beta_{j-1}; m_{j-1}), (s_i\beta_{j+1}; m_{j+1}), \dots , (s_i\beta_{t}; m_{t})  \} .\]
In this case, we have $\beta_j = \alpha_i$, and obtain $m(\mathfrak p , ws_i) =m(\phi_i(\mathfrak p), w)- 1$. 
\end{proof}

\begin{proposition} If $\lambda\in Q^-$ and $w \in W$, then
\begin{equation}\label{eq.dlambdaformula}
(-1)^{\ell(w)} H(-w \circ \lambda) =  q^{\ell(w)} \ \sum_{\mathfrak p \in \mathcal P(-\lambda)} (-q)^{m(\mathfrak p,w)} .
\end{equation}
\end{proposition}

\begin{proof}
Write $w=s_{i_1} s_{i_2}\cdots s_{i_\ell}$ as a reduced expression. 
By Lemma \ref{lem-bij-mphi}, we have
\begin{align*}
& H(-s_{i_1}s_{i_2} \cdots s_{i_\ell} \circ \lambda)  =  \sum_{\mathfrak p \in \mathcal P(-s_{i_1} s_{i_2}\cdots s_{i_\ell} \circ \lambda)} (-q)^{m(\mathfrak p, \mathrm{id})} \\& =  \sum_{\mathfrak p \in \mathcal P(-s_{i_2} \cdots s_{i_\ell} \circ \lambda)} (-q)^{m(\phi_{s_{i_1}}(\mathfrak p), \mathrm{id})} =  \sum_{\mathfrak p \in \mathcal P(-s_{i_2} \cdots s_{i_\ell} \circ \lambda)} (-q)^{m(\mathfrak p, s_{i_1})+1} \\ & =  \sum_{\mathfrak p \in \mathcal P(-s_{i_3} \cdots s_{i_\ell} \circ \lambda)} (-q)^{m(\phi_{s_{i_2}}(\mathfrak p), s_{i_1})} =  \sum_{\mathfrak p \in \mathcal P(-s_{i_3} \cdots s_{i_\ell} \circ \lambda)} (-q)^{m(\mathfrak p, s_{i_1}s_{i_2})+2} \\ &= \cdots = \sum_{\mathfrak p \in \mathcal P(- \lambda)} (-q)^{m(\mathfrak p, s_{i_1}\cdots s_{i_\ell})+\ell},
\end{align*}
which amounts to the identity \eqref{eq.dlambdaformula}.
\end{proof}

\section{Polynomiality}\label{section.poly}

In this section we prove Theorem \ref{thm.polynomial}. That is, we show that $d_\lambda$ is a polynomial when the Weyl group $W$ of $\mathfrak g$ is a universal Coxeter group. 

\medskip

Assume that $W$ be a universal Coxeter group of rank $n\in\mathbb{Z}_{>0}$. By definition, the group $W$ is isomorphic to the free product of $n$-copies of $\mathbb Z /2 \mathbb Z$. Denote its generators by $s_i$, $i=1,\dots,n$. The identity element is the only word of length $0$, and for any $\ell\geq1$ there are $n(n-1)^{\ell-1}$ words of length $\ell$. We thus compute the Poincar{\'e} series of $W$ to be:
\[
\chi(q)=\sum_{w\in W}q^{\ell(w)}=1+nq\sum_{k=0}^{\infty}\left((n-1)q\right)^k=1+\frac{nq}{1-(n-1)q}=\frac{1+q}{1-(n-1)q}\in\mathbb{Z}[[q]].
\]
Given $\lambda\in{ P^+\cap} Q^-_{\rm im}$, our aim is to establish the polynomiality of 
\begin{equation} \label{dla}
d_{\lambda}=\frac{\sum_{w\in W}(-1)^{\ell(w)}H(w\circ\lambda)}{\chi(q)}=\frac{\left(1-(n-1)q\right)\sum_{w\in W}(-1)^{\ell(w)}H(w\circ\lambda)}{1+q}\in\mathbb{Z}[[q]].
\end{equation}
Fix $\lambda\in{ P^+\cap} Q^-_{\rm im}$, and define
\begin{equation} \label{defN}
N= \max \{ \ell(w)+1 : (w\alpha_{i}; j) \text{ is a part of any } \p \in \mathcal P(-\lambda) \text{ for } 1\leq i\leq n, \ j \in \mathbb Z \text{ and } w \in W \}.
\end{equation}
Since $W$ is a universal Coxeter group, the number $N$ is well-defined.
For the time being, fix an arbitrary element $v \in W$ of length $N$. Let  $W_v$ be the set of elements in $W$ whose reduced word has $v$ as its rightmost factor. For $\mathfrak{p}\in\mathcal{P}(-\lambda)$, define
\begin{equation}
m(\p, W_v) = |\p|-2  \times \#\left\{ (\beta_i; j) \in \mathfrak p \, : \,  w\beta_i<0 \text{  for some $w \in W_v$}\right\}. 
\end{equation}
Write
\begin{equation}\label{ak}
\sum_{\p \in \mathcal P(-\lambda)} (-q)^{m(\p,W_v)} = \sum_{k=0}^r a_k q^k 
\end{equation}
for some $a_k \in \mathbb Z$ and $r \ge 0$,
and define
\begin{align} Q_{v}&:=\sum_{k=0}^{r-1} \left( (n-1)^k a_0+(n-1)^{k-1}a_1+(n-1)^{k-2}a_2+\cdots + a_k \right ) q^k,  \label{eqQ} \\
A_v&:=(n-1)^r a_0+(n-1)^{r-1}a_1+(n-1)^{r-2}a_2+\cdots + a_r. \label{eqA} 
\end{align}
It follows from \eqref{eq.dlambdaformula} that
\begin{align*}
\sum_{w \in W_v} (-1)^{\ell(w)} H(-w \circ \lambda) &= \sum_{w \in W_v} q^{\ell(w)} \ \sum_{\p \in \mathcal P(-\lambda)} (-q)^{m(\p,w)} \\&
=  q^N \left ( Q_v  +  A_v \, q^r\left ( 1 + (n-1) q + (n-1)^2 q^2 + \cdots \right ) \right)\\ &= q^N \left ( Q_v  + A_v \, \frac {q^r} {1 -(n-1)q }  \right) ,
\end{align*}
and so
\begin{multline}\label{analyticcont}
\sum_{w \in W} (-1)^{\ell(w)} H(-w \circ \lambda) \\  
= \sum_{w \in W, \ell(w) <N} q^{\ell(w)} \sum_{\p \in \mathcal P(-\lambda)} (-q)^{m(\p,w)} + q^N  \sum_{v \in W, \ell(v) =N} \left ( Q_v  + A_v \frac {q^r} {1 -(n-1)q }  \right) .
\end{multline}

\begin{proposition} \label{prop-vanishing}
For any $w \in W$, the sum  
\[ \sum_{w \in W} (-1)^{\ell(w)} H(-w \circ \lambda)   \]
is divisible by $1+q$.

\end{proposition}

\begin{proof}
Let $N$ be defined as in \eqref{defN}, and $v \in W$ be an arbitrary element of length $N$. For any $w\in W$, we have
\[\sum_{\p\in\mathcal{P}(-\lambda)}1^{m(\p,w)}=|\mathcal{P}(-\lambda)|.\]
Therefore
\[\sum_{\substack{w\in W\\\ell(w)<N}}(-1)^{\ell(w)}\sum_{\p\in\mathcal{P}(-\lambda)}1^{m(\p,w)}=|\mathcal{P}(-\lambda)| \sum_{\substack{w\in W\\\ell(w)<N}}(-1)^{\ell(w)}.\]
Since
\[\sum_{\substack{w\in W\\\ell(w)<N}}(-1)^{\ell(w)}=\sum_{k=0}^{N-1}(-1)^k\#\{w\in W:\ell(w)=k\},\]
and
\[\#\{w\in W:\ell(w)=k\}=\begin{cases}1,&k=0\\ n(n-1)^{k-1}, & k>0, \end{cases}\]
we deduce
\begin{align*}
\sum_{\substack{w\in W\\\ell(w)<N}}(-1)^{\ell(w)} &=1-n+n(n-1)-n(n-1)^2+\cdots+(-1)^{N-1}n(n-1)^{N-2}\\
&=1-n\left[1+(-1)(n-1)+(-1)^2(n-1)^2+\cdots+(-1)^{N-2}(n-1)^{N-2}\right]\\
&=1-n\left[1+(1-n)+(1-n)^2+\cdots+(1-n)^{N-2}\right]\\
&=1-n\left(\frac{1-(1-n)^{N-1}}{1-(1-n)}\right)=(1-n)^{N-1}.
\end{align*}
Combining the above, we see that
\[\sum_{\substack{w\in W\\\ell(w)<N}}(-1)^{\ell(w)}\sum_{\p\in\mathcal{P}(-\lambda)}1^{m(\p,w)}=(1-n)^{N-1} |\mathcal{P}(-\lambda)|.\]
Let $A_v$ and $Q_v$ be defined as in \eqref{eqA} and \eqref{eqQ}, respectively.
Then  we have
\[\frac{(-1)^r}{n}A_v=\frac{(-1)^r}{n}\left((n-1)^ra_0+(n-1)^{r-1}a_1+\cdots+a_r\right),\]
and 
\begin{multline*}
Q_v=a_0\left(1-(n-1)+(n-1)^2+\cdots+(-1)^{r-1}(n-1)^{r-1}\right)\\
-a_1\left(1-(n-1)+(n-1)^2+\cdots+(-1)^{r-2}(n-1)^{r-2}\right)
+\cdots+(-1)^{r-1}a_{r-1}\\
=\tfrac{1-(1-n)^r}{n}\, a_0-\tfrac{1-(1-n)^{r-1}}{n}\, a_1+\cdots+(-1)^{r-2} \, \tfrac{1-(1-n)^2}{n} \, a_{r-2}+(-1)^{r-1}a_{r-1},
\end{multline*}
so that
\[Q_v+A_v\frac{(-1)^r}{n}=\frac{1}{n}\left(a_0-a_1+a_2-a_3+\cdots+(-1)^ra_r\right)=\frac{1}{n}|\mathcal{P}(-\lambda)|.\]
Evaluating~\eqref{analyticcont} at $q=-1$, we get
\[(1-n)^{N-1} |\mathcal{P}(-\lambda)|+(-1)^Nn(n-1)^{N-1}\frac{1}{n} | \mathcal{P}(-\lambda)| =0.\]
\end{proof}

\begin{proof}[Proof of Theorem~\ref{thm.polynomial}] It follows from \eqref{analyticcont} that 
\[ ({1 -(n-1)q } ) \sum_{w \in W} (-1)^{\ell(w)} H(-w \circ \lambda) \] is a polynomial. By Proposition \ref{prop-vanishing}, the sum $\sum_{w \in W} (-1)^{\ell(w)} H(-w \circ \lambda)$ is divisible by $1+q$. Thus we see from \eqref{dla} that $d_\lambda$ is a polynomial. 
\end{proof}

\begin{remark}
From \cite[(3-21)]{Kim--Lee}, we know that 
\begin{equation}\label{eq.KL321}
H(-w\circ \lambda;-1)= H(\rho - w(\lambda+\rho); -1) = \dim V(\rho)_{w(\lambda+\rho)} = \dim V(\rho)_{\lambda+\rho} .
\end{equation}
Taking the alternating sum,  we get
\[ \sum_{w \in W} (-1)^{\ell(w)} H(-w\circ \lambda;-1) = \dim V(\rho)_{\lambda+\rho} \sum_{w \in W} (-1)^{\ell(w)}, \]
which does not converge. In Proposition~\ref{prop-vanishing}, the sum $\sum_{w \in W} (-1)^{\ell(w)}H(-w\circ\lambda;q)$ is to be interpreted via its analytic continuation given by the rational function in equation~\eqref{analyticcont}.

\end{remark}

\appendix
\section{~}\label{section.rank2}
In this appendix, we consider the explicit example of the Kac--Moody algebra $\mathfrak{g}=\mathcal{H}(3)$ associated to the generalized Cartan matrix 
\begin{equation}\label{eq.Rank2Cartan}
A=\begin{pmatrix}2&-3 \\ -3 &2 \end{pmatrix}.
\end{equation} 
The Weyl group $W$ is the universal Coxeter group of rank $2$, that is,  $W$ is isomorphic to the free product $\left(\mathbb{Z}/2\mathbb{Z}\right)\ast\left(\mathbb{Z}/2\mathbb{Z}\right)$. As there are two elements for a given length $\ge 1$, the Poincar\'{e} series has the following closed form:
\[
\chi(q)=1+2q\sum_{\ell=0}^{\infty}q^{\ell}=1+2q\left(\frac{1}{1-q}\right)=\frac{1+q}{1-q}\in\mathbb{Z}[[q]].
\]
We denote the simple roots by $\alpha_1,\alpha_2$ and the simple reflections by $s_1,s_2$ as before.
When $\mathfrak p=\{ (\beta_1; m_1), (\beta_2; m_2), \dots , (\beta_t; m_t)  \}$ is an admissible partition, we will sometimes write
\[ \mathfrak p = (\beta_1; m_1)+ (\beta_2; m_2)+ \dots + (\beta_t; m_t) .\]
For the root multiplicities of $\mathcal H(3)$, we refer the reader to \cite{KaMe}. 
\begin{ex}\label{ex.212} 
Consider $\lambda=-2\alpha_1-2\alpha_2\in Q^-$. Then there are 4 admissible partitions of $-\lambda$:
\begin{enumerate}
\item $(2\alpha_1+2\alpha_2;1)$,
\item $(\alpha_1;1)+(\alpha_1+2\alpha_2;1)$,
\item $(\alpha_2;1)+(2\alpha_1+\alpha_2;1)$,
\item $ (\alpha_1;1)+(\alpha_2;1)+(\alpha_1+\alpha_2;1)$.
\end{enumerate}
All the roots appearing in the list above have multiplicity $1$, and so
\[
H(-\lambda) =H(2\alpha_1+2\alpha_2)=-q + 2 q^2 -q^3 = -q(q-1)^2.
\] 
We calculate
\[
- s_1\circ\lambda= - s_1 (\lambda+ \rho) + \rho=5\alpha_1+2\alpha_2,
\]
and see that $-s_1\circ\lambda$ has 4 admissible partitions:
\begin{enumerate}
\item $(5\alpha_1+2\alpha_2;1)$,
\item $(\alpha_1;1)+(4\alpha_1+2\alpha_2;1)$,
\item $(2\alpha_1+\alpha_2;1)+(3\alpha_1+\alpha_2;1)$,
\item $ (\alpha_1;1)+(\alpha_1+\alpha_2;1)+(3\alpha_1+\alpha_2;1)$.
\end{enumerate}
Again, all the roots appearing have multiplicity $1$. We therefore deduce that
\[
H(-s_1\circ\lambda)=-q(q-1)^2=H(-\lambda).
\]
Similarly, we compute
\[
H(-s_2\circ\lambda)=-q(q-1)^2=H(-\lambda).
\]
The circle action of $s_1s_2$ on $\lambda$ yields
\[
-s_1s_2\circ\lambda=14\alpha_1+5\alpha_2,
\]
which is not a root. Yet again we have 4 admissible partitions, but the lengths are different:
\begin{enumerate}
\item $(\alpha_1;1)+(13\alpha_1+5\alpha_2;1)$, 
\item $(\alpha_1;1)+(3\alpha_1+\alpha_2;1)+(10\alpha_1+4\alpha_2;1)$, 
\item $(\alpha_1;1)+(5\alpha_1+2\alpha_2;1)+(8\alpha_1+3\alpha_2;1)$, 
\item $(\alpha_1;1)+(2\alpha_1+\alpha_2;1)+(3\alpha_1+\alpha_2;1)+(8\alpha_1+3\alpha_2;1)$,
\end{enumerate}
in which  all the roots still have multiplicity $1$. It follows that
\[
H(-s_1s_2\circ\lambda)=q^2(q-1)^2=-qH(-s_1\circ\lambda)=-qH(-\lambda).
\]
One can see that this pattern continues as proved in the previous section  to yield
\begin{equation} \label{eqn-hwcirc}  
H(-w \circ \lambda) = (-q)^{\ell(w)-1} H(-\lambda), \quad w \in W, \ w \neq \mathrm{id}. 
 \end{equation}
It follows from equations~\eqref{eq.formulad} and~\eqref{eqn-hwcirc} that
\begin{align*} 
\sum_{w \in W} (-1)^{\ell(w)} H(-w \circ \lambda)  &= H(-\lambda) +\sum_{w \in W, \ w \neq \mathrm {id}} (-1)^{\ell(w)} (-q)^{\ell(w)-1} H(- \lambda)\\
&= (1+q^{-1}) H(-\lambda) - q^{-1} \chi(q) H(-\lambda) ,
\end{align*}
and so 
\[ 
d_{-2\alpha_1 -2 \alpha_2} = \left [ (1+q^{-1}) \, \frac {1-q}{1+q} -q^{-1} \right ] H(-\lambda) = - H(-\lambda)  =q(q-1)^2. 
\] 
\end{ex}

\begin{ex}
Let $\lambda=-2\alpha_1-3\alpha_2 \in Q^-$, which is a root with multiplicity 2. We have admissible partitions:
\begin{enumerate}
\item $(2\alpha_1+3\alpha_2,n)$, $n\in\{1,2\}$,
\item $(2\alpha_1+2\alpha_2,1)+(\alpha_2,1)$,
\item $(\alpha_1+3\alpha_2,1)+(\alpha_1,1)$,
\item $(\alpha_1+2\alpha_2,1)+(\alpha_1+\alpha_2,1)$,
\item $(\alpha_1+2\alpha_2,1)+(\alpha_1,1)+(\alpha_2,1)$.
\end{enumerate}
Therefore
\[H(-\lambda)=-2q+3q^2-q^3=-q(q-1)(q-2).\]
We have
\[-s_1\circ\lambda=8\alpha_1+3\alpha_2,\]
which is a root with multiplicity 1, and admissible partitions:
\begin{enumerate}
\item $(8\alpha_1+3\alpha_2,1)$,
\item $(7\alpha_1+3\alpha_2,n)+(\alpha_1,1)$, $n\in\{1,2\}$,
\item $(5\alpha_1+2\alpha_2,1)+(3\alpha_1+\alpha_2,1)$,
\item $(5\alpha_1+2\alpha_2,1)+(2\alpha_1+\alpha_2,1)+(\alpha_1,1)$,
\item $(4\alpha_1+2\alpha_2,1)+(3\alpha_1+\alpha_2,1)+(\alpha_1,1)$.
\end{enumerate}
Therefore
\[H(-s_1\circ\lambda)=-q+3q^2-2q^3=-q(q-1)(2q-1).\]
On the other hand
\[-s_2\circ\lambda=2\alpha_1+4\alpha_2,\]
which is a root of multiplicity 1, and admissible partitions:
\begin{enumerate}
\item $(2\alpha_1+4\alpha_2,1)$, 
\item $(2\alpha_1+3\alpha_2,n)+(\alpha_2,1)$, $n\in\{1,2\}$,
\item $(\alpha_1+3\alpha_2,1)+(\alpha_1+\alpha_2,1)$,
\item $(\alpha_1+3\alpha_2,1)+(\alpha_1,1)+(\alpha_2,1)$,
\item $(\alpha_1+2\alpha_2,1)+(\alpha_1+\alpha_2,1)+(\alpha_2,1)$.
\end{enumerate}
Therefore:
\[H(-s_2\circ\lambda)=-q+3q^2-2q^3.\]
Now 
\[-s_2s_1\circ\lambda=8\alpha_1+22\alpha_2,\]
which is not a root. We make a list of all admissible partitions:
\begin{enumerate}
\item $(8\alpha_1+21\alpha_2,1)+(\alpha_2,1)$,
\item $(7\alpha_1+18\alpha_2,n_1)+(\alpha_1+3\alpha_2,1)+(\alpha_2,1)$, $n_1\in\{1,2\}$,
\item $(5\alpha_1+13\alpha_2,1)+(3\alpha_1+8\alpha_2,1)+(\alpha_2,1)$, 
\item $(5\alpha_1+13\alpha_2,1)+(2\alpha_1+5\alpha_2,1)+(\alpha_1+3\alpha_2,1)+(\alpha_2,1)$,
\item $(4\alpha_1+10\alpha_2,1)+(3\alpha_1+8\alpha_2,1)+(\alpha_1+3\alpha_2,1)+(\alpha_2,1)$.
\end{enumerate}
It follows that
\[H(-s_2s_1\circ\lambda)=q^2-3q^2+2q^4=q^2(2q^2-3q+2).\]
On the other hand,
\[ - s_1s_2\circ\lambda=11\alpha_1+4\alpha_2,\]
which is not a root, and its admissible  partitions are:
\begin{enumerate}
\item $(10\alpha_1+4\alpha_2,1)+(\alpha_1,1)$,
\item $(8\alpha_1+3\alpha_2,1)+(3\alpha_1+\alpha_2,1)$,
\item $(8\alpha_1+3\alpha_2,1)+(2\alpha_1+\alpha_2,1)+(\alpha_1,1)$,
\item $(7\alpha_1+3\alpha_2,n)+(3\alpha_1+\alpha_2)+(\alpha_1,1)$, $n\in\{1,2\}$,
\item $(5\alpha_1+2\alpha_2,1)+(3\alpha_1+2\alpha_2,1)+(2\alpha_1+\alpha_2,1)+(\alpha_1,1)$.
\end{enumerate}
Therefore:
\[H(-s_1s_2\circ\lambda)=2q^2-3q^3+q^4=q^2(q-1)(q-2).\]
Next
\[-s_1s_2s_1\circ\lambda=59\alpha_1+22\alpha_2,\]
which is not a root. We have 
\[
H(-s_1s_2s_1\circ\lambda)=-q^3+3q^4-2q^5.
\]
Also
\[-s_2s_1s_2\circ\lambda=11\alpha_1+30\alpha_2,\]
and
\[
H(-s_2s_1s_2\circ\lambda)=-2q^3+3q^4-q^5.
\]
We arrange the information above into a table, in which the columns are indexed by $n\in\mathbb{N}$ and the rows are indexed by $w\in W$ (written as reduced words, ordered lexicographically). The entry corresponding to row $w$ and column $n$ is the coefficient of $q^n$ in $H(-w\circ\lambda)$. An empty space indicates that the coefficient is zero. There is one additional column, which lists the image $w\circ\lambda$ of $\lambda$ under the circle action by $w$, written in coordinates with respect to the basis $\{-\alpha_1,-\alpha_2\}$. 

\begin{center}
\begin{tabular}{|c||c||c|c|c|c|c|c|}
\hline
$w$ & $w \circ \lambda$ &  1 & 2 & 3 & 4 & 5 & $\cdots$ \\
\hline \hline
$\mathrm{id}$&  $(2,3)$ &$-2$ & $3$ & $-1$ & & & \\
$s_1$&$(8,3)$ & $-1$ & $3$ & $-2$ &&&\\ 
$s_2$&$(2,4)$ & $-1$ & $3$ & $-2$ &&& \\
$s_2s_1$&$(8,22)$ && $1$ & $-3$ & $2$ && \\
$s_1s_2$&$(11,4)$  && $2$ & $-3$ & $1$ && \\
$s_1s_2s_1$& $(59,22)$ &&& $-1$ & $3$ & $-2$ & \\
$s_2s_1s_2$& $(11,30)$ &&& $-2$ & $3$ & $-1$ & \\
$\vdots$& $\vdots$ & & & & $\vdots$ & $\vdots$ & $\ddots$ \\
\hline
\end{tabular}
\end{center}
Observe that the strings $(-2,3,-1)$ and $(-1,3,-2)$ repeat with each iteration, shifting 1 space and switching signs as the word length increases.  The coefficient of $q^n$ in $\chi(q) \, d_{\lambda}$ can be calculated by taking the sum of the entries in a column multiplied by $(-1)^{\ell(w)}$. We see that
\[
d_{-2\alpha_1-3\alpha_2}=0.
\]
\end{ex}

\begin{ex}\label{ex.312}
Consider $\lambda=-3\alpha_1-3\alpha_2\in Q^-$. There are 12 admissible partitions of $-\lambda$:
\begin{enumerate}
\item $(3 \alpha_1+3 \alpha_2; 1)$, 
\item $(3 \alpha_1+3 \alpha_2; 2)$, 
\item $(3 \alpha_1+3 \alpha_2; 3)$,  
\item $(3 \alpha_1+2 \alpha_2; 1) + (\alpha_2; 1)$, 
\item $(3 \alpha_1+2 \alpha_2; 2) + (\alpha_2;1)$, 
\item $(2 \alpha_1+3 \alpha_2; 1) + (\alpha_1; 1)$, 
\item $(2 \alpha_1+3 \alpha_2; 2) + (\alpha_1;1)$, 
\item $(2 \alpha_1+2 \alpha_2; 1) + (\alpha_1+\alpha_2; 1)$, 
\item $(2 \alpha_1+ \alpha_2; 1) + (\alpha_1+2\alpha_2;1)$, 
\item $(\alpha_1; 1)+ ( \alpha_1+\alpha_2; 1) + (\alpha_1+2\alpha_2; 1)$, 
\item $ (\alpha_2; 1)+ ( \alpha_1+\alpha_2; 1) + (2\alpha_1+\alpha_2; 1)$, 
\item $(\alpha_1; 1)+ ( \alpha_2; 1) + (2\alpha_1+2\alpha_2; 1)$.
\end{enumerate}
Note that $m(3\alpha_1+3\alpha_2)=3$, $m(2\alpha_1+3\alpha_2)=m(3\alpha_1+2\alpha_2)=2$, and the other roots each have multiplicity 1. We conclude
\[
H(-\lambda) =H(3\alpha_1+3\alpha_2)=-3q + 6 q^2 -3q^3 = -3q(q-1)^2.
\] 
We continue to obtain the following table.

\begin{center}
\begin{tabular}{|c||c||c|c|c|c|c|c|c|c|}
\hline 
$w$&$w\circ\lambda$&1&2&3&4&5&6&$\cdots$\\
\hline \hline
id&(3,3)&$-3$&6&$-3$&&&&\\
$s_1$&(7,3)&$-2$&5&$-4$&1&&&\\
$s_2$&(3,7)&$-2$&5&$-4$&1&&&\\
$s_1s_2$&(19,7)&&2&$-5$&4&$-1$&&\\
$s_2s_1$&(7,19)&&2&$-5$&4&$-1$&&\\
$s_1s_2s_1$&(19,51)&&&$-2$&$5$&$-4$&$1$&\\
$s_2s_1s_2$&(51,19)&&&$-2$&$5$&$-4$&$1$&\\
$\vdots$&$\vdots$&&&&$\vdots$&$\vdots$&$\vdots$&$\ddots$\\
\hline
\end{tabular}
\end{center}
Observe that the string $(-2,5,-4,1)$ repeats with each iteration, shifting 1 space and switching signs as the word length increases. Since $-2+5-4+1=0$, the coefficient of $q^n$ is $0$ for $n\geq 4$. As it happens, the coefficient of $q^2$ is $0$ too. We conclude
\[d_{-3\alpha_1-3\alpha_2}=\frac{-q^3+q}{\chi(q)}=q(q-1)^2.\] 
\end{ex}
\begin{ex}\label{ex.3142}
Let $\lambda=-3\alpha_1-4\alpha_2\in Q^-$. We produce a table similar to that in Example~\ref{ex.312}
\begin{center}
\begin{tabular}{|c||c||c|c|c|c|c|c|c|c|}
\hline
$w$&$w\circ\lambda$&1&2&3&4&5&6&7&$\cdots$\\
\hline \hline
id&$(3,4)$&$-4$&8&$-5$&1&&&&\\
$s_1$&$(10,4)$&$-1$&7&$-8$&2&&&&\\
$s_2$&$(3,6)$&$-3$&$8$&$-6$&$1$&&&&\\
$s_2s_1$&(10,27)&&$1$&$-7$&$8$&$-2$&&&\\
$s_1s_2$&(16,6)&&4&$-9$&5&&&&\\
$s_1s_2s_1$&(72,27)&&&$-1$&7&$-8$&2&&\\
$s_2s_1s_2$&(16,43)&&&$-4$&9&$-5$&&&\\
$s_1s_2s_1s_2$&(72,190)&&&&1&$-7$&8&$-2$&\\
$s_2s_1s_2s_1$&(114,43)&&&&4&$-9$&5&&\\
$\vdots$&$\vdots$&&&&&$\vdots$&$\vdots$&$\vdots$&$\ddots$\\
\hline
\end{tabular}
\end{center}
This time, the strings $(1, -7, 8, -2)$ and $(4,-9,5)$ alternate.  
Note that both strings sum to zero. We see that  for $n\geq5$, the coefficient of $q^n$ in $d_{\lambda}$ is $0$. The coefficients of $q$ and $q^4$ are also  $0$.   Altogether we obtain
\[d_{-3\alpha_1-4\alpha_2}=\frac{-2q^2(1+q)}{\chi(q)}=2q^2(q-1).\]
\end{ex}

\begin{ex}
We may compute other  $d_\lambda$'s in a similar way. In the following table, the entry in the space $(m,n)$ is the polynomial $d_{\lambda}$ for $\lambda=-m\alpha_1-(m+n)\alpha_2 \in P^+$.  From symmetry in $\mathcal H(3)$, we have $d_{-m\alpha_1-(m+n)\alpha_2}= d_{-(m+n)\alpha_1-m \alpha_2}$.
\begin{center}
\begin{tabular}{|c||c|c|c|c|c|c|c|}
\hline 
 & 0 & 1 &2&3&4 \\
\hline \hline
0& 1 & $-q(q-1)$ & $q(q-1)^2$  & $q(q-1)^2$ & $2q(q-1)^2$  \\
\hline
1&  &  & $0$ & $2q^2(q-1)$ & $-q^2(q-1)(q-4)$  \\
\hline
2&  &  & &  & $-q (q - 1)^2 (q^2 + q - 1)$  \\
\hline
\end{tabular}
\end{center}
We also have
\[ d_{-5\alpha_1-5\alpha_2} = -q(q-1)(q^3+3q^2-7q+2) .\]

\end{ex}

\end{document}